\documentclass[12pt]{amsart}
\usepackage{hyperref,enumerate}
\hypersetup{
    pdftitle=   {Vertex-minors and Erd\H{o}s-Hajnal conjecture},
   pdfauthor=  {Maria Chudnovsky and Sang-il Oum}
}
\newcommand\abs[1]{\lvert #1\rvert}
\newtheorem{THM}{Theorem}

\newtheorem{COR}[THM]{Corollary}
\theoremstyle{remark}
\newtheorem*{REM}{Remark}

\begin{document}
\title{Vertex-minors and the~Erd\H{o}s-Hajnal~conjecture}
\author{Maria Chudnovsky}
\address[Chudnovsky]{Department of Mathematics, Princeton University, Princeton, USA}
\email{mchudnov@math.princeton.edu}
\author{Sang-il Oum}
\address[Oum]{Department of Mathematical Sciences, KAIST, Daejeon, South Korea}
\email{sangil@kaist.edu}
\thanks{
Chudnovsky was supported by NSF grant DMS-1550991. This material is based upon work supported in part by the U. S. Army Research Laboratory and the U. S. Army Research Office under grant number W911NF-16-1-0404.
Oum was supported by the National Research Foundation of Korea (NRF) grant funded by the Korea government (MSIT) (No. NRF-2017R1A2B4005020).}
\date{\today}

\begin{abstract}
  We prove that for every graph $H$,
  there exists $\varepsilon>0$ such that
  every $n$-vertex graph with no vertex-minors isomorphic to $H$
  has a pair of disjoint sets $A$, $B$ of vertices such that 
  $\abs{A}, \abs{B}\ge \varepsilon n$ and 
  $A$ is complete or anticomplete to $B$.
  We deduce this from
  recent work of Chudnovsky, Scott, Seymour, and Spirkl (2018).
  This proves the analog of the Erd\H{o}s-Hajnal conjecture for vertex-minors.
\end{abstract}
\keywords{vertex-minor, Erd\H{o}s-Hajnal conjecture}

\maketitle

For a graph $G$, let $\alpha(G)$ be the maximum size of an independent set, that
is a set of pairwise non-adjacent  vertices.
Let $\omega(G)$ be the maximum size of a clique, that is a set of pairwise
adjacent vertices.
In 1989, Erd\H{o}s and Hajnal~\cite{EH1989} conjectured that for every graph $H$,
there exists $\varepsilon>0$ such that
if a graph $G$ has no induced subgraph isomorphic to $H$,
then \[\max(\omega(G),\alpha(G))\ge\abs{V(G)}^\varepsilon.\]
A few years ago, Chudnovsky proposed a weaker question; is it true if we replace ``induced subgraphs'' by ``vertex-minors''?

If a class $\mathcal G$ of graphs closed under taking induced subgraphs has
some $\varepsilon>0$ such that every graph in $\mathcal G$ has
an independent set or a clique of size more than  $\abs{V(G)}^\varepsilon$, then
we say that $\mathcal G$ has the \emph{Erd\H{o}s-Hajnal property}.

We prove that for every graph $H$, the class of graphs with no vertex-minor isomorphic to $H$ has the Erd\H{o}s-Hajnal property. 
In addition, we prove a stronger property that is defined as follows.
A set $A$ of vertices is \emph{complete} to a set $B$ of vertices
if every vertex in $A$ is adjacent to every vertex of $B$.
A set $A$ of vertices is \emph{anticomplete} to a set $B$ of vertices
if every vertex in $A$ is  non-adjacent to every vertex of $B$.
If a class $\mathcal G$ of graphs closed under taking induced subgraphs has
some $\varepsilon>0$ such  that
every graph in $\mathcal G$ has  a complete or anticomplete pair of disjoint sets $A$, $B$ with $\abs{A}, \abs{B}\ge \varepsilon\abs{V(G)}$,
then we say that $\mathcal G$ has the \emph{strong Erd\H{o}s-Hajnal property}.
It is well known that
the strong Erd\H{o}s-Hajnal property implies the Erd\H{o}s-Hajnal property, see~\cite{APPRS2005,FP2008a}.
We prove that for every graph $H$, the class of graphs with no vertex-minor isomorphic to $H$
has the strong Erd\H{o}s-Hajnal property.

Before presenting our theorem, we state the definition of vertex-minors \cite{Oum2004}.
For a graph $G$ and its vertex $v$,
the \emph{local complementation} at $v$
results in the new graph, denoted by $G*v$, such that
$V(G*v)=V(G)$
and
two distinct vertices $x$, $y$ are adjacent in $G*v$
if either
\begin{enumerate}[(i)]
\item both $x$ and $y$ are neighbors of $v$ in $G$  and $x$, $y$ are non-adjacent in $G$,
  or
\item at least one of $x$ or $y$ is non-adjacent to $v$ in $G$
  and $x$, $y$ are adjacent in $G$.
\end{enumerate}
A graph $H$ is a \emph{vertex-minor} of a graph $G$ if
$H$ is an induced subgraph of $G*v_1*v_2*\cdots*v_k$
for some sequences of vertices $v_1,v_2,\ldots,v_k$ (not necessarily distinct) with $k\ge 0$.

Now we state our main theorem.
\begin{THM}\label{thm:main}
  For every graph $H$, there exists $\varepsilon>0$ such that
  every $n$-vertex graph $G$
  has a vertex-minor isomorphic to $H$ 
  or
  has a pair of disjoint sets $A$, $B$ of of vertices such that
  $A$ is either complete or anticomplete to $B$
  and $\abs{A},\abs{B}\ge \varepsilon n$.
\end{THM}

As the strong Erd\H{o}s-Hajnal property implies the Erd\H{o}s-Hajanl property, 
we deduce the following.
\begin{COR}
  For every graph $H$, there exists $\varepsilon>0$ such that
  if a graph $G$ has no vertex-minor isomorphic to $H$, then
  \[ \max(\alpha(G),\omega(G))\ge\abs{V(G)}^\varepsilon.\]
\end{COR}

Now let us present the proof. 
Our proof is based on the following theorems of Chudnovsky, Scott, Seymour, and Spirkl~\cite{CSSS2018}.
\begin{THM}[Chudnovsky et al.~\cite{CSSS2018}]\label{thm:gen}
  For every graph $H$, there exists $c>0$ such that
  every graph $G$
  has an induced subgraph isomorphic
  to a subdivision of $H$ or the complement of a subdivision of $H$
  or
  has a pair of disjoint sets $A$, $B$ of of vertices such that
  $A$ is either complete or anticomplete to $B$
  and $\abs{A},\abs{B}\ge c\abs{V(G)}$.
\end{THM}
\begin{THM}[Chudnovsky et al.~\cite{CSSS2018}]\label{thm:sparse}
  For every graph $H$, there exists $\delta>0$ such that
  every $n$-vertex graph $G$ with $\abs{E(G)}\le \delta\abs{V(G)}^2$
  has an induced subgraph isomorphic
  to a subdivision of $H$ 
  or
  has an anticomplete pair of disjoint sets $A$, $B$ of of vertices such that
  $\abs{A},\abs{B}\ge\delta n$.
\end{THM}
\begin{proof}[Proof of Theorem~\ref{thm:main}]
  Let $c$, $\delta$ be the constants given by Theorems~\ref{thm:gen} and \ref{thm:sparse}.
  We claim that $\varepsilon=\min(2c\delta,\delta)$.

  If $G$ has an induced subdivision of $H$,
  then we can apply local complementations to degree-$2$ vertices to obtain
  a vertex-minor isomorphic to $H$, contradicting our assumption.
  Thus $G$ has no induced subdivision of $H$.
  By the same reason, $G*v$ has no induced subdivision of $H$
  for every vertex $v$.

  If every vertex of $G$ has degree at most $2\delta n$,
  then
  $\abs{E(G)}\le \delta n^2$. By Theorem~\ref{thm:sparse}, $G$ has an anticomplete
  pair of disjoint sets $A$, $B$ with $\abs{A},\abs{B}\ge \delta n$.

  If a vertex $v$ has degree more than $2\delta n$,
  then let $G'$ be the subgraph of $G$ induced by all neighbors of $v$.
  Note that neither $G'$ nor the complement of $G'$ has an induced subdivision of $H$
  and therefore by Theorem~\ref{thm:gen}, $G'$ has an anticomplete or complete
  pair of sets $A$, $B$ with $\abs{A},\abs{B}\ge c\abs{V(G')}>2c\delta n$.
\end{proof}

\begin{REM}
There are two major examples of graph classes known to be closed
under taking vertex-minors; graphs of rank-width at most $k$~\cite{Oum2004}
and circle graphs~\cite{Bouchet1994}. It is easy to see that the class of graphs of rank-width at most $k$ has the strong Erd\H{o}s-Hajnal property. To see this, observe that an $n$-vertex graph $G$ of rank-width at most $k$ has a vertex set $X$ such that the cut-rank of $X$ is at most $k$ and $\abs{X},\abs{V(G)}-\abs{X}>n/3$.
Then one can partition each of  $X$ and $V(G)-X$ into at most $2^k$ subsets such that
each part of $X$ is complete or anticomplete to each part of $V(G)-X$. This proves that such a graph has an anticomplete or complete pair of sets $A$, $B$ such that $\abs{A}, \abs{B}>(n/3)/2^k$.
The class of circle graphs has the strong Erd\H{o}s-Hajnal property, implied
by a theorem of Pach and Solymosi~\cite{PS2001}.
\end{REM}

\subsection*{Acknowledgement}
This paper was written during the 2018 Barbados Graph Theory workshop at Bellairs Research Institute, McGill University. 
The authors would like to thank Vaidyanathan Sivaraman, who
suggested this problem during the workshop.

\end{document}